\documentclass[]{llncs}

\usepackage{amsmath,amssymb}
\usepackage{comment}
\usepackage{bbold}

\newcommand{\Z}{\mathbb{Z}}
\newcommand{\N}{\mathbb{N}}
\newcommand{\lcm}{\mathrm{lcm}}
\newcommand{\F}{\mathcal{F}}
\newcommand{\Con}{\mathrm{Con}}
\newcommand{\Sub}{\mathrm{Sub}}

\begin{document}

\title{On Shift Spaces with Algebraic Structure\thanks{Research supported by the Academy of Finland Grant 131558}}

\author{Ville Salo \and Ilkka T\"orm\"a}

\maketitle

\begin{abstract}
We investigate subshifts with a general algebraic structure and cellular automata on them, with an emphasis on (order-theoretic) lattices. Our main results concern the characterization of Boolean algebraic subshifts, conditions for algebraic subshifts to be recoded into cellwise algebras and the limit dynamics of homomorphic cellular automata on lattice subshifts.
\end{abstract}


\section{Introduction}

There has been considerable interest in the connections between groups and symbolic dynamics in the literature. Linear cellular automata have been investigated extensively \cite{ItOsNa83} \cite{CaFoMaMa00} \cite{MaMa98} \cite{Sa97} \cite{Ka00}, and some authors \cite{Sc95} \cite{Ki87} \cite{BoSc08} have looked into subshifts with a group structure. Since the study of linear automata and group subshifts has proven fruitful, it seems natural to consider whether other algebraic structures produce interesting behavior in subshifts and CA. This paper is, as far as we are aware, the first foray in this direction, with an emphasis on order-theoretic lattices.

In the first section, we look at the case of giving a subshift the structure of an order-theoretic lattice, with the operations defined cellwise. We characterize all such subshifts in terms of the local rules that define them. In particular, we find that a result of Kitchens \cite{Ki87} for group subshifts also holds for Boolean algebras, namely that they are, up to conjugacy, products of a full shift and a finite subshift. In our case, the conjugacy can even be made algebraic.

In the second section, we find necessary and sufficient conditions for when an algebra operation can be recoded to one defined cellwise. The conditions are stated in terms of the affine maps of the algebra, that is, functions built from the algebra operations. We prove that the recoding can be done when all the maps are `local', in the sense that their radius is bounded as cellular automata. Many natural classes of algebras, in particular groups and Boolean algebras, always satisfy this condition.

In the third section, we consider cellular automata that are also algebra homomorphisms, in the case that the algebra has the so-called congruence-product property (lattices exhibit this property). We prove that all such CA are stable, and the limit set is conjugate to a product of full shifts and finite shifts.

\section{Definitions}

If $S$ is a finite set, the \emph{alphabet}, the space $S^\Z$ is called the \emph{full shift over $S$}. If $x \in S^\Z$, we denote the $i$th coordinate of $x$ with $x_i$, and abbreviate $x_i x_{i+1} \cdots x_{i+n-1}$ by $x_{[i,i+n-1]}$. For a word $w \in S^n$, we say that $w$ \emph{occurs in $x$}, if there exists $i$ such that $w = x_{[i,i+n-1]}$. On $S^\Z$ we define the \emph{shift map} $\sigma_S$ (or simply $\sigma$, if $S$ is clear from the context) by $\sigma_S(x)_i = x_{i+1}$ for all $i$.

We assume a basic acquaintance with symbolic dynamics, in particular the notions of subshift, forbidden word, SFT, sofic shift, mixingness, block map and cellular automaton. A clear exposition can be found in \cite{LiMa95}. The set of words of length $n$ occurring in a subshift $X$ is denoted by $B_n(X)$.

Let $T$ be a set of pairs $(f,n)$, where $f$ is a symbol and $n \in \N$, with the property that $(f,m), (f,n) \in T \Rightarrow m = n$. Here, $n$ is called the \emph{arity} of $f$. An \emph{algebra of type $T$} is a pair $(S,F)$, where $S$ is a set and for each $(f,n) \in T$ we have a function $f' : S^n \to S$ in the set $F$. In this case, we identify $f'$ with $f$, and the functions $f$ are called \emph{algebra operations}. We usually identify $(S,F)$ with $S$, if $F$ is clear from the context.

A \emph{variety of type $T$ defined by the set of identities $I$} is the class of algebras of type $T$ that satisfy the identities in $I$. We do not define identities rigorously, but are satisfied with an intuitive presentation (see the examples below). If $\F$ is a variety, $(S,F) \in \F$ and $R \subset S$ is closed under the operations of $S$, then we call $(R,F)$ a \emph{subalgebra} of $S$. The set of subalgebras of $S$ is denoted $\Sub(S)$. A function $g : S \to R$ between two algebras in $\F$ is called a \emph{homomorphism} or an \emph{$\F$-morphism} if $g(f(s_1,\ldots,s_n)) = f(g(s_1),\ldots,g(s_n))$ for each $n$-ary operation $f$. If $g$ is bijective, it is called an \emph{isomorphism}. The \emph{direct product} of an indexed family $(S_i)_{i \in \mathcal{I}}$ of algebras is the algebra $\prod_{i \in \mathcal{I}} S_i$, where the operations are defined cellwise ($f(s^1, \ldots, s^n)_i = f(s^1_i, \ldots, s^n_i)$). An algebra is \emph{directly indecomposable}, if it is not isomorphic to a product of two nontrivial algebras. All finite algebras are isomorphic to a finite product of directly indecomposable finite algebras. All varieties are closed under subalgebras, homomorphic images and products \cite{BuSa81}.

If $\sim \; \subset S \times S$ is an equivalence relation that satisfies
\[ s_1 \sim t_1, \ldots, s_n \sim t_n \Longrightarrow f(s_1,\ldots,s_n) \sim f(t_1,\ldots,t_n) \]
for all $s_i,t_i \in S$ and all $n$-ary operations $f$, we say that $\sim$ is a \emph{congruence} on $S$. The set of congruences on an algebra $S$ is denoted $\Con(S)$. A natural algebraic structure is induced by the algebra operations on the set of equivalence classes $S/\!\sim$, the kernel $\ker(g) = \{(s,t) \in S \times S \;|\; g(s) = g(t)\}$ of a homomorphism $g$ is always a congruence, and $S/\ker(g)$ is isomorphic to $g(S)$ \cite{BuSa81} (the last claim is known as the \emph{Homomorphism Theorem}).

The variety of \emph{lattices} has type $\{(\wedge,2),(\vee,2)\}$ and is defined by the identities
\[
\begin{array}{rl}
x \wedge x \approx x, & x \wedge y \approx y \wedge x \enspace , \\
(x \wedge y) \wedge z \approx x \wedge (y \wedge z), & x \wedge (x \vee y) \approx x \enspace ,
\end{array}
\]
and the same identities with $\wedge$ and $\vee$ interchanged (their \emph{dual versions}). The operations $\wedge$ and $\vee$ are called \emph{meet} and \emph{join}, respectively. It is known that the variety of lattices coincides with the class of partially ordered sets where all pairs of elements have suprema and infima, where the correspondence is given by $x \wedge y = \inf \{x,y\}$ and $x \vee y = \sup \{x,y\}$. If $S$ is a lattice and $a,b \in S$, then we denote $[a,b] = \{c \in S \;|\; a \leq c \leq b\}$, where $\leq$ is the partial order of $S$. A lattice $S$ is \emph{modular} if it satisfies the additional constraint that whenever $a,b,c \in S$ and $a \leq b$, then $a \vee (b \wedge c) = b \wedge (a \vee c)$. The variety of \emph{distributive lattices} satisfies the lattice identities and the additional identity
\[ x \vee (y \wedge z) = (x \vee y) \wedge (x \vee z) \]
and its dual version. All distributive lattices are modular \cite{BuSa81}. Of particular interest is the binary lattice $\mathbb{2}$ containing the elements $\{0, 1\}$ with their usual numerical order.

The variety of \emph{Boolean algebras} has type $\{(\wedge,2),(\vee,2),(\, \bar{\cdot},1),(1,0),(0,0)\}$. A Boolean algebra is a distributive lattice w.r.t. $\wedge$ and $\vee$, and also satisfies
\[
\begin{array}{rl}
x \wedge 0 \approx 0, & x \vee 1 \approx 1 \enspace , \\
x \wedge \bar{x} \approx 0, & x \vee \bar{x} \approx 1 \enspace .
\end{array}
\]
It is known that every finite Boolean algebra is isomorphic to the algebra of subsets $2^T$ of some set $T$ where the ordering is given by set inclusion \cite{BuSa81}.

Let $\F$ be a variety of algebras. We call $X \subset S^\Z$ an \emph{$\F$-subshift}, if it is a subshift and has an algebra structure in $\F$ whose operations are block maps. In particular, every finite $S \in \F$ induces a natural cellwise algebra structure on $S^\Z$ by $f(x^1,\ldots,x^n)_i = f(x^1_i,\ldots,x^n_i)$ for all $i$ (that is, $S^\Z$ is taken to be a direct product), and in this case, if $X \subset S^\Z$ is a subshift in $\Sub(S^\Z)$, it is called a \emph{cellwise $\F$-subshift}. A conjucacy that is also an $\F$-morphism is called \emph{algebraic}. A cellular automaton with state set $S$ that is also an $\F$-morphism is called \emph{$\F$-linear}.

\section{Cellwise Lattice Subshifts}

In this section, let $S$ be a finite lattice ordered by $\leq$.

We begin by giving a characterization for the cellwise lattice subshifts.

\begin{definition}
Let $X \subset S^\Z$ be a subshift and $a \in S$. For such $X$ we define $m^a$ and $M^a$ by
\[ \forall i \in \Z: m^a_i = \bigwedge_{\substack{x \in X \\ x_0 \geq a}} x_i,\; M^a_i = \bigvee_{\substack{x \in X \\ x_0 \leq a}} x_i \enspace . \]
\end{definition}

It is easy to see that $a \leq b$ implies $m^a \leq m^b$ and $M^a \leq M^b$.

\begin{lemma}
\label{lemma:FirstCharacterization}
Let $X \subset S^\Z$ be a subshift with $B_1(X) = S$. The following are equivalent:
\begin{itemize}
\item $X$ is a cellwise lattice subshift.
\item For all $x \in S^\Z$, we have $x \in X$ iff $x \geq \sigma^{-i}(m^{x_i})$ holds for all $i$.
\item For all $x \in S^\Z$, we have $x \in X$ iff $x \leq \sigma^{-i}(M^{x_i})$ holds for all $i$.
\end{itemize}
\end{lemma}

\begin{proof}
Assume first that $X$ is a cellwise lattice subshift. Let $a \in S$, and for all $i \in \Z$, let $x^i \in X$ such that $x^i_i = m^a_i$. Such $x^i$ exist, since $X \in \Sub(S^\Z)$ and has alphabet $S$. Then the sequence $\bigwedge_{i=-n}^n x^i$ approaches $m^a$ as $n$ grows, so $m^a \in X$. The second condition is then necessary by the definition of the $m^{x_i}$, and it is sufficient since the sequence $\bigvee_{i=-n}^n \sigma^{-i}(m^{x_i})$ approaches $x$ as $n$ grows.

Assume then that $X \notin \Sub(S^\Z)$. Suppose first that there exist $x,y \in X$ with $x \wedge y \notin X$, and let $a_i = x_i \wedge y_i$. Now, for all $i$, we have that $x \wedge y \geq \sigma^{-i}(m^{x_i}) \wedge \sigma^{-i}(m^{y_i}) \geq \sigma^{-i}(m^{a_i})$, so that the second condition does not hold.

Suppose then that $x \vee y \notin X$, and let $a_i = x_i \vee y_i$. Then we have $x \vee y \geq \sigma^{-i}(m^{x_i})$ and $x \vee y \geq \sigma^{-i}(m^{y_i})$, so that $x \vee y \geq \sigma^{-i}(m^{a_i})$, and again the second condition fails.

The third condition is dual to the second, and is proven similarly. \qed
\end{proof}

The following proposition characterizes sofic cellwise lattice subshifts.

\begin{proposition}
\label{prop:SoficLattices}
A cellwise lattice subshift $X$ is sofic if and only if all the $m^a$ are eventually periodic in both directions (and the dual claim for $M^a$ holds as well).
\end{proposition}

\begin{proof}
First, assume a cellwise defined lattice $X \subset S^\Z$ is sofic, and let $G$ be a right-resolving presentation for it \cite{LiMa95}. To prove eventual periodicity of $m_a$ for $a \in S$, we take a state $s_0 \in G$ such that $s$ may occur in the origin in a bi-infinite path with label $m_a$. Due to right-resolvingness of the presentation, for each $i > 0$, there is a unique state $s_i$ which can be reached in a presentation of $m_a$ with $s_0$ in the origin. Clearly, for all $i > 0$, $s_i$ must be the state reached from $s_{i-1}$ having the minimal label in the lattice $S$. This means that $s_i$ is a function of $s_{i-1}$, which in particular implies that $m_a$ is eventually periodic to the right. By considering a left-resolving presentation, we see that $m_a$ is also periodic to the left. There is a dual proof for eventual periodicity of the $M_a$.

Now, consider the case where all the $m_a$ are eventually periodic in both directions. It is enough to show that the language of $X$ regular \cite{LiMa95}. By using the fact that two-directional alternating finite automata have regular languages (see Lemma~3.11 in \cite{Sa11}), this is very easy: Our automaton universally quantifies over all positions $i \in [1, |w|]$ and directions $d \in \{l, r\}$. For $d = r$, the automaton checks that $w_{[i, |w|]} \geq m^{w_i}_{[0, |w| - i]}$ (cellwise), and the other direction is checked symmetrically. The final step is possible due to our assumption on $m^{w_i}$ and can be done with a one-way DFA. \qed
\end{proof}

\begin{lemma}
\label{lemma:BinaryLatticeSubshifts}
A nontrivial subshift $X \subset \mathbb{2}^\Z$ is a cellwise lattice subshift if and only if one of the following holds:
\begin{itemize}
\item $X = S^\Z$.
\item $X = \{x \in S^\Z \;|\; x = \sigma^n(x)\}$ for some $n$.
\item $X = \{x \in S^\Z \;|\; \forall i \in \Z, p \in P: x_i = 1 \Rightarrow x_{i + p} = 1\}$ for a finite set $P \subset \N$.
\item $X = \{x \in S^\Z \;|\; \forall i \in \Z, p \in P: x_i = 1 \Rightarrow x_{i - p} = 1\}$ for a finite set $P \subset \N$.
\end{itemize}
\end{lemma}

\begin{proof}
Let $X$ be a cellwise lattice subshift, and let $K = \{i \in \Z \;|\; m^1_i = 1\}$. If $K = \{0\}$, then $X = S^\Z$, so suppose that $K \neq \{0\}$. From Lemma~\ref{lemma:FirstCharacterization} and the fact that $m^1 \in X$ it follows that $K$ is a subsemigroup of $\Z$, i.e.,
\[ \{ a_1 n_1 + \cdots + a_k n_k \;|\; a_i \in \N, n_i \in K \} \subset K \enspace . \]
A standard pigeonhole argument then shows that $K$ is generated by some finite set $P \subset \Z$. The third and fourth items correspond to the cases $P \subset \N$ and $-P \subset \N$, respectively.

If $P$ contains both positive and negative elements, let $n = \gcd P$. An application of Bezout's identity then shows that $K = \{pn \;|\; p \in \Z\}$, and the second case follows. \qed
\end{proof}

\begin{corollary}
Every cellwise lattice subshift of $\mathbb{2}^\Z$ is an SFT.
\end{corollary}

\begin{corollary}
\label{corollary:BinaryComplements}
If $X$ is also closed under complementation in Lemma~\ref{lemma:BinaryLatticeSubshifts}, only the first two cases can occur.
\end{corollary}

\begin{proof}
Assume the contrary, say case $3$. Then $X$ contains the point $x = {}^\infty01^\infty$. But now $\bar{x} = {}^\infty10^\infty \in X$, and thus $m^1 = {}^\infty 010^\infty$. \qed
\end{proof}

We have now completely characterized all binary Boolean algebraic subshifts.

\begin{definition}
Let $S = 2^T$ be a finite Boolean algebra and $X \subset S^\Z$ a Boolean algebraic subshift. For each $t \in T$, define the block map $\pi_t : X \to \{0,1\}^\Z$ by
\[ \pi_t(x)_i =
\left\{\begin{array}{ll}
1, & \mbox{if } t \in x_i \\
0, & \mbox{if } t \notin x_i
\end{array}\right. \enspace . \]
If $R \subset T$, we also define $\pi_R : X \to (\{0,1\}^R)^\Z$ by $\pi_R(x) = (\pi_t(x))_{t \in R}$.
\end{definition}

Note that each $\pi_t$ and $\pi_R$ above is also a homomorphism of Boolean algebras, and that $\pi_T$ is an injective block map from $X$ to $(\{0,1\}^T)^\Z$.

\begin{definition}
Let $S = 2^T$ be a finite Boolean algebra. A subshift $X \subset S^\Z$ is called \emph{simple} if there exists a set $R \subset T$ such that the following conditions hold.
\begin{itemize}
\item For every $t \in T$, we have either $\pi_t(X) = \{ x \in \mathbb{2}^\Z \;|\; x = \sigma^n(x) \}$ for some $n \in \N$, or $\pi_t(X) = \mathbb{2}^\Z$.
\item For all $t \in T$, there exists $r \in R$ and $k \in \Z$ with $\pi_t = \sigma^k \circ \pi_r$.
\item The elements of $R$ are independent in the sense that if $r \in R$, $x \in \pi_r(X)$ and $y \in \pi_{R - \{r\}}(X)$, then there exists $z \in X$ with $\pi_r(z) = x$ and $\pi_{R - \{r\}}(z) = y$.
\end{itemize}
\end{definition}

Informally, a subshift is simple if it is essentially a product of full shifts and periodic subshifts. Note that in particular, such a subshift is conjugate to a product of a full shift and a finite shift. Clearly, all simple shifts are cellwise Boolean algebraic subshifts. The converse also holds:

\begin{theorem}
\label{theorem:BooleansAreNice}
Let $S = 2^T$ be a finite Boolean algebra and $X \subset S^\Z$ a subshift with $B_1(X) = S$. If $X \in \Sub(S^\Z)$, then $X$ is simple.
\end{theorem}

\begin{proof}
Since each $\pi_t(X)$ is a cellwise Boolean algebraic subshift over $\mathbb{2}$, the first condition follows from Corollary~\ref{corollary:BinaryComplements}.

Define an equivalence relation $\sim$ on $T$ by $t \sim s$ if $\pi_t = \sigma^k \circ \pi_s$ for some $k \in \Z$. Let $R$ be a set of representatives for the equivalence classes of $\sim$. Then for this $R$, the second condition holds.

Consider then the configuration $m^{\{r\}}$ for some $r \in R$. If $t \in m^{\{r\}}_i$ for any $t \in T$ and $i \in \Z$, then $\pi_t(x) \geq \sigma^{-i}(\pi_r(x))$, and by complementing, $\pi_t(x) \leq \sigma^{-i}(\pi_r(x))$, holds for all $x \in X$. The converse is also true, which means that $R \cap m^{\{r\}}_i \subset \{r\}$ for all $i$. The third condition then follows from Lemma~\ref{lemma:FirstCharacterization}, since $X$ is also a cellwise lattice subshift. \qed
\end{proof}

\begin{example}
Let $S$ be a lattice with at least $3$ elements, and let $a \neq 1$ be a least successor of $0$. Consider the rule `if $x_i = 1$, then $x_{i \pm 2^j} \geq a$ for all $j \in \N$'. By Lemma~\ref{lemma:FirstCharacterization}, the subshift generated by this rule forms a lattice. It is easy to see that this subshift is mixing, but is not even sofic.
\end{example}

\section{General Algebraic Subshifts}

In this section, $\F$ is a variety of algebras. We now consider $\F$-subshifts $X$ that are not necessarily cellwise. One way to define such general shifts is to use a cellwise $\F$-subshift $Y$ and a conjugacy $\phi : X \to Y$, and define
\[ f(x_1, \ldots, x_n) = \phi^{-1}(f(\phi(x_1), \ldots, \phi(x_n))) \]
for all $n$-ary algebra operations $f$ of $\F$. Clearly, $\phi$ now becomes an algebraic conjugacy. The main theorems in this section address the issue of deciding whether a given $\F$-subshift is algebraically conjugate to some cellwise $\F$-subshift.

\begin{definition}
An \emph{affine map} of an algebra $S$ is inductively defined as either $t(\xi) = \xi$ (the identity map), $t(\xi) = a$ for some $a \in S$ (the constant map) or $t(\xi) = f(a_1, \ldots, a_n)$, where $f$ is an $n$-ary operation, one of the $a_i$ is an affine map and the rest are constants $a_j \in S$. Here, $\xi \notin S$ is used as a variable. To each affine map $t$ we also associate a function $\mathrm{Eval}(t) : S \to S$ by replacing $\xi$ with the function argument and evaluating the resulting expression. We may denote $\mathrm{Eval}(t)(a)$ by $t(a)$ if $a \in S$. The set of affine maps of $S$ is denoted by $\mathrm{Af}(S)$.
\end{definition}

For example, in the ring $\Z$, the term $t(\xi) = 2 \cdot (3 + (\xi \cdot (-4)))$ is an affine map, and $\mathrm{Eval}(t)(i) = -8 i + 6$ for all $i \in \Z$.

The following is a dynamical characterization of $\F$-subshifts that are cellwise up to algebraic conjugacy. The proof uses a common recoding technique found, for example, in the Recoding Construction 4.3.1 of \cite{Ki98}.

\begin{theorem}
\label{theorem:UskomatonAsia}
Let $X \subset S^\Z$ be an $\F$-subshift. Then there exists a cellwise $\F$-subshift $Y$ and an algebraic conjugacy $\phi : X \to Y$ if and only if there is an $r$ such that for all $t \in \mathrm{Af}(X)$, the block map $\mathrm{Eval}(t)$ has radius at most $r$.
\end{theorem}

\begin{proof}
Suppose first that such an $r$ exists. Then, we can also meaningfully apply an affine map $t \in \mathrm{Af}(X)$ to a word $w \in B_{2r+1}(X)$ rooted at the origin by extending it arbitrarily to a configuration $x \in X$, and taking the center cell of $t(x)$. We define the following equivalence relation on $B_{2r+1}(X)$:
\[ \forall v, w \in B_{2r+1}(X): v \sim w \iff \forall t \in \mathrm{Af}(X): t(v)_0 = t(w)_0 \enspace . \]
Note that, in particular, $v \sim w \implies v_0 = w_0$. We define an injective block map $\psi : X \to (B_{2r+1}(X)/\sim)^\Z$ by $\psi(x)_i = x_{[i-r,i+r]}/\sim$, and denote $Y = \psi(X)$. We denote the obtained conjugacy by $\phi : X \to Y$.

In order to make the algebra operations commute with $\phi$, we define
\[ f(y_1, \ldots, y_n) = \phi(f(\phi^{-1}(y_1), \ldots, \phi^{-1}(y_n))) \]
for all $n$-ary algebra operations $f$, which is obviously well-defined. Now $\phi$ extends to a bijection between $\mathrm{Af}(X)$ and $\mathrm{Af}(Y)$ in a natural way. Let us show that every algebra operation $f$ is then defined cellwise in $Y$. Consider two points $y, y' \in Y$ with $y_0 = y'_0$. We need to show that $\phi(t)(y)_0 = \phi(t)(y')_0$ for all $\phi(t) \in \mathrm{Af}(Y)$. Assume the contrary, that $\phi(t)(y)_0 \not= \phi(t)(y')_0$ for some $\phi(t) \in \mathrm{Af}(Y)$. Let $x = \phi^{-1}(y)$ and $x' = \phi^{-1}(y')$. Then also $t(x)_{[-r,r]} = v \not\sim w = t(x')_{[-r,r]}$, and thus there exists $t' \in \mathrm{Af}(X)$ such that $t'(v) \not= t'(w)$. But by the assumption on affine maps, $t'' = t' \circ t$ has radius $r$. Now we have $t''(x)_0 \not= t''(x')_0$, which is a contradiction, since $x_{[-r,r]} \sim x'_{[-r,r]}$.

For the converse, note that if $X$ is algebraically conjugate to a cellwise $\F$-subshift $Y$ via the conjugacy $\phi$, then the radius of every translate is at most the sum of the radii of $\phi$ and $\phi^{-1}$. \qed
\end{proof}

We also obtain a sufficient algebraic condition for algebraic conjugacy with a cellwise $\F$-subshift. In the special case of the full shift, this becomes a characterization. We start with a definition.

\begin{definition}
We define the \emph{depth} of an affine map as the number of nested algebra operations in it. We say an algebra $S$ is \emph{$k$-shallow} if for every $t \in \mathrm{Af}(S)$ there exists $t' \in \mathrm{Af}(S)$ of depth at most $k$ such that $\mathrm{Eval}(t) = \mathrm{Eval}(t')$.
\end{definition}

\begin{theorem}
\label{theorem:ShallowImpliesPointwise}
Let $X \subset S^\Z$ be an $\F$-subshift. If $X$ is $k$-shallow, then it is cellwise up to algebraic conjugacy. If $X$ is algebraically conjugate to $R^\Z$ where $R \in \F$, then $X$ is $k$-shallow for some $k$.
\end{theorem}

\begin{proof}
If $X$ is $k$-shallow, then clearly all affine maps have uniformly bounded radii, and Theorem~\ref{theorem:UskomatonAsia} gives the result.

For the other claim, it suffices to show that $R^\Z$ is $k$-shallow for some $k$. Since $R$ is finite, the set $\Gamma = \{\mathrm{Eval}(t) \;|\; t \in \mathrm{Af}(R)\}$ is finite. For an affine map $t \in \mbox{Af}(R^\Z)$ we denote by $t_i$ the affine map in $\mathrm{Af}(R)$ that $t$ computes in coordinate $i$. For each affine map $t \in \mbox{Af}(R^\Z)$ we define $\Delta_t = \{\mathrm{Eval}(t_i) \;|\; i \in \Z\}$. Let $n = |\Gamma|$ and note that since $R^n$ is finite, it is $k$-shallow for some $k$.

Let $t \in \mathrm{Af}(R^\Z)$ and let $j_1, \ldots, j_n$ be coordinates such that for all $h \in \Delta_t$ we have $\mathrm{Eval}(t_{j_i}) = h$ for some $i$. Construct an affine map $s \in \mathrm{Af}(R^n)$ by $s_i = t_{j_i}$. Since $R^n$ is $k$-shallow we find some affine map $s' \in \mathrm{Af}(R^n)$ of depth at most $k$ with $\mathrm{Eval}(s) = \mathrm{Eval}(s')$. We may now define an affine map $t' \in \mathrm{Af}(R^\Z)$ by $t'_i = s'_{j'_i}$ for all $i$, where $j'_i$ is such that $\mathrm{Eval}(t_i) = \mathrm{Eval}(s_{j'_i})$. Since $\mathrm{Eval}(s_j) = \mathrm{Eval}(s'_j)$ for all $j$, we have $\mathrm{Eval}(t) = \mathrm{Eval}(t')$. \qed
\end{proof}

\begin{corollary}
Up to algebraic conjugacy, every distributive lattice, Boolean algebra, ring, semigroup, monoid and group subshift is defined cellwise.
\end{corollary}

\begin{proof}
By finding suitable normal forms, one easily sees that distributive lattices, semigroups and monoids are $2$-shallow, while Boolean algebras, groups and rings are $3$-shallow. Note that affine maps have at most one unknown variable. \qed
\end{proof}

\begin{corollary}
Every Boolean algebraic subshift is algebraically conjugate to a product of a full shift and a finite shift.
\end{corollary}

\begin{proof}
This follows from the previous corollary and Theorem~\ref{theorem:BooleansAreNice}. \qed
\end{proof}

The previous corollary is an analogue of a result of Kitchens \cite{Ki87} for group subshifts.

\begin{example}
In the proof of Theorem~\ref{theorem:ShallowImpliesPointwise}, for the claim that algebraic conjugacy implies $k$-shallowness, the use of a full shift is crucial: the claim is false even for the mixing one-step SFT $X \subset \{0,1,2,\bot\}^\Z$ with the forbidden pairs $\{10, 11, 20, 21\}$ equipped with the cellwise binary operation
\[ \begin{array}{c|cccc}
\cdot & 0 & 1 & 2 & \bot \\ \hline
0 & 0 & 0 & 0 & \bot \\
1 & 1 & 2 & 1 & \bot \\
2 & 2 & 2 & 2 & \bot \\
\bot & \bot & \bot & \bot & \bot \\
\end{array} \]

It is easy to see that $X$ is indeed a mixing groupoid subshift, when the application of $\cdot$ is defined pointwise. For a coordinate $i \in \Z$, denote by $x^i$ the unique point of $X$ with $x^i_i = 1$ in which $\bot$ does not appear. Now, for all $k \in \N$, define the translation $t_k$ by 
\[ t_k(\xi) = (\cdots ((\xi \cdot x^1)\cdot x^2)\cdots \cdot x^k), \]
that is, we multiply from the right with the elements $x^i$ for all $1 \leq i \leq k$.

Consider the action of $t_k$ on the points $x^i$. If $1 \leq i \leq k$, then we see that $t_k(x^i)$ contains no $1$, and conversely, if $i < 1$ or $k < i$, then $t_k(x^i) = x^i$. It follows from the definition of $\cdot$ that $t_k$ cannot be realized with a translation whose constants contain the symbol $\bot$, and from this it is easily seen that $t_k$ is not equivalent to any translation of depth less than $k$.
\end{example}

Since lattices are not shallow in general, Theorem~\ref{theorem:UskomatonAsia} and Theorem~\ref{theorem:ShallowImpliesPointwise} do not tell us much about lattice subshifts. In fact, the following example gives a lattice subshift on a mixing SFT which is not cellwise up to algebraic conjugacy.

\begin{example}
Let $S = \{1_+,1_-,0_+,0_-\}$ have the lattice structure $1_+>0_+>0_-$ and $1_+>1_->0_-$, and define $X \subset S^\Z$ as the SFT with forbidden words $0_-1_+$ and $a_-b_-c_\delta$ for all $(a,b,c)\neq(0,0,0)$ and $\delta \in \{+,-\}$. Define $X$ to have a lattice structure where $\vee$ is defined cellwise, and $\wedge$ is defined as first applying the cellwise meet of $S^\Z$, and then rewriting instances of $0_-1_+$ to $0_-0_+$ and instances of $a_-b_-c_\delta$ to $0_-0_-0_\delta$. An easy calculation confirms that the operations are well-defined and $X$ indeed forms a lattice subshift.

We now show that $X$ has affine maps of arbitrary radius by giving two right asymptotic points $x,y \in X$ with $x_0 \neq y_0$, and then constructing affine maps $t_k \in \mathrm{Af}(X)$ for all $k \in \N$ such that $t_k(x)_k \neq t_k(y)_k$. Let $x = {}^\infty 1_+ . 0_+ {1_+}^\infty$ and $y = {}^\infty {1_+}^\infty$. Define $z_k = \sigma^{-k}({}^\infty 1_+ . 1_- {1_+}^\infty)$ and $z' = {}^\infty {0_+}^\infty$, and let
\[ t_k(\xi) = (( \cdots (((\xi \wedge z_0) \vee z') \wedge z_1) \vee z' \cdots ) \wedge z_k) \vee z' \enspace . \]
It is easy to see that $t_k(x) = {}^\infty 1_+ . (0_+)^{k+2}{1_+}^\infty$, and on the other hand, $t_k(y) = y$.
\end{example}

We are not aware of a modular lattice which is not shallow. This raises the natural question whether every modular lattice subshift defined cellwise up to algebraic conjugacy.

That all semigroups have shallow affine maps is due to associativity, which can be used to combine constants in terms. \emph{Quasigroups} are a generalization of groups, where the requirement of associativity is dropped. The variety of quasigroups is of type $\{(\cdot,2),(/\,,2),(\backslash\,,2)\}$ and has the identities
\[
\begin{array}{rl}
x \cdot (x \backslash\, y) \approx y, & x \backslash (x \cdot y) \approx y, \\
(x /\, y) \cdot y \approx x, & (x \cdot y) /\, y \approx x.
\end{array}
\]
We show that, in general, quasigroups do not have shallow translations, and that they are not necessarily cellwise up to algebraic conjugacy either, even if defined on a full shift.

\begin{example}
Let $Q=\{0,1\}$ be the group defined by addition modulo $2$. We can make $Q^\Z$ into a quasigroup by defining $x \cdot y = \sigma(x) + \sigma(y)$, $x /\, y = \sigma^{-1}(x) + y$ and $x \backslash\, y = x + \sigma^{-1}(y)$. Here, $+$ is defined cellwise. It is easy to see that the translations defined by $t_i(\xi) = (\cdots(\xi \cdot 0) \cdots \cdot 0)$, that is, multiplying by $0$ from the right $i$ times, have arbirarily large radii. Thus $Q^\Z$ can't be conjugate to a cellwise defined quasigroup.
\end{example}

\section{$\F$-linear Cellular Automata}

Let $\F$ be again a variety, and $S$ a finite member of $\F$. We only consider cellwise defined algebraic structures of $S^\Z$ in this section. We begin with the following useful lemma, characterizing all $\F$-linear cellular automata.

\begin{lemma}
\label{lemma:LinearityCharacterization}
A CA $G$ with neighborhood radius $r$ is $\F$-linear if and only if its local rule $g$ is an $\F$-morphism from $S^{2r + 1}$ to $S$.
\end{lemma}

\begin{proof}
Let $f$ be an $n$-ary algebra operation of $\F$.

Assume first that $G$ is $\F$-linear. Let $w_i \in S^{2r + 1}$ for all $i \in [1,n]$, and let $s \in S$ be arbitrary. Then 
\[ G(f({}^\infty sw_1s^\infty, \ldots, {}^\infty sw_ns^\infty)) = f(G({}^\infty sw_1s^\infty), \ldots, G({}^\infty sw_ns^\infty)), \]
so in particular
\[ g(f(w_1, \ldots, w_n)) = f(g(w_1), \ldots, g(w_n)) \]
for the local function $g$. Thus $g$ is an $\F$-morphism.

On the other hand, if the local function $g$ is an $\F$-morphism, consider arbitrary points $x^1, \ldots, x^n \in S^\Z$. We have
\begin{align*}
G(f(x^1, \ldots, x^n))_i =& g(f(x^1_{[i-r, i+r]}, \ldots, x^n_{[i-r, i+r]})) \\
=& f(g(x^1_{[i-r, i+r]}), \ldots, g(x^n_{[i-r, i+r]})) \\
=& f(G(x^1)_i, \ldots, G(x^n)_i)
\end{align*}
for all $i \in \Z$, which implies that $G$ is $\F$-linear.
\end{proof}

\begin{definition}
The variety $\F$ has the \emph{congruence-product property}, if for all finite families $(S_i)_{i \in [1,n]}$ of algebras in $\F$ we have that
\[ \Con(\prod_{i=1}^n S_i) = \prod_{i=1}^n \Con(S_i) \enspace . \]
\end{definition}

A proof of the following can be found, for example, in \cite{Gr71}.

\begin{lemma}
\label{lemma:CongruenceCharacterization}
The variety of lattices has the congruence-product property.
\end{lemma}

In the remainder of this section, we show that if $\F$ has the congruence-product property, then the $\F$-linear cellular automata have very simple limit sets and limit dynamics. In particular, by the above lemma, our results hold for lattice-linear automata.

\begin{lemma}
\label{lemma:EverythingCharacterized}
Let $\F$ be a variety with the congruence-product property and $S \in \F$ finite. Let $G$ be an $\F$-linear CA on $S^\Z$ with radius $r$ and local function $g$. Denote by $R \subset S$ the alphabet of the limit set of $G$ (which is clearly a subalgebra), and denote by $\pi'_k$ the canonical projections $R^{2r + 1} \to R$. Let $\prod_{i=1}^m R_i$ be the decomposition of $R$ into directly indecomposable algebras, and denote by $\pi_i$ the canonical projections $R \to R_i$. Then for each $i \in [1,m]$ there exist $j_i$, $k_i$ and a surjective $\F$-morphism $h_i : R_{j_i} \to R_i$ such that $\pi_i \circ g|_{R^{2r + 1}} = h_i \circ \pi_{j_i} \circ \pi'_{k_i}$.
\end{lemma}

\begin{proof}
We may assume that $R = S$, so that $g$ is already surjective. Denote $g_i = \pi_i \circ g$, and decompose the domain $R^n = \prod_{j=1}^n \prod_{k=1}^m R_k$ into the directly indecomposable algebras $R_k$. Now $\ker g_i$ is a congruence, and since $g_i(R^n) = R_i$ is directly indecomposable, the homomorphism theorem states that $R^n / \ker g_i$ must be directly indecomposable. Since $\F$ has the congruence-product property, we have that $\ker g_i = \prod_{j=0}^{nm} \sim_j\, \in \prod_{j=1}^n \prod_{k=1}^m \Con(R_k)$, and now only one of these $\sim_j$ can be nontrivial. Thus $g_i$ is of the desired form.
\end{proof}

Since the proof of Lemma~\ref{lemma:CongruenceCharacterization} can also be carried out for Boolean algebras, we have the following.

\begin{corollary}
If $S = 2^T$ is a Boolean algebra and $G$ a Boolean-linear CA on $S^\Z$, then for each $t \in T$, either $\pi_t(G(S^\Z))$ is trivial, or we have $i \in \Z$ and $t' \in T$ such that $\pi_t \circ G = \pi_{t'} \circ \sigma^i$.
\end{corollary}

\begin{theorem}
Let $\F$ be a variety with the congruence-product property and $S \in \F$ finite. The limit set $X$ of an $\F$-linear cellular automaton $G$ on $S^\Z$ is algebraically conjugate to a product of full shifts, and $G$ is stable. Furthermore, there exists $p \in \N$ such that $G^p|_X$ is a product of powers of shift maps.
\end{theorem}

\begin{proof}
Let $\prod_{i=1}^m S_i$ be the decomposition of $S$ into directly indecomposable algebras. Define $H = (\{S_1, \ldots, S_m\},E)$ as the directed graph where $(S_i,S_j) \in E$ iff the domain of the surjective map $h_j$ given by Lemma~\ref{lemma:EverythingCharacterized} is $S_i$. Since each $S_i$ has exactly one incoming arrow, every strongly connected component of $H$ is a cycle or a single vertex. Let $S_i$ be in a cycle, say $S_i \to S_{i_1} \to \cdots \to S_{i_{p'-1}} \to S_i$. Since $S_i$ is finite, the map $f_i = h_i \circ h_{i_{p'-1}} \circ \cdots \circ h_{i_1}$ is an automorphism of $S_i$, and there exists $p_i \in \N$ such that $f_i^{p_i}$ is the identity map of $S_i$. This in turn implies that for all $x \in S^\Z$, we have
\[ \pi_i(G^{p_i}(x)) = \pi_i(\sigma^{k_i}(x)) \]
for some $k_i \in \Z$. That is, $G^{p_i}$ simply shifts the $S_i$-components of points by a constant amount. Let $\mathcal{I}$ be the set of indices $i$ such that $S_i$ occurs in a cycle, and let $p = \lcm_{i \in \mathcal{I}}(p_i)$. Clearly, $G$ has a natural reversible restriction on the full shift $S_\mathcal{I}^\Z$, where $S_\mathcal{I} = \prod_{i \in \mathcal{I}} S_i$.

Consider then $S_j$ for some $j \in \mathcal{J} = [1, m] - \mathcal{I}$. By following the incoming arrows we necessarily find an $i(j) \in \mathcal{I}$ and a path of the form
\[ S_{i(j)} \to S_{i_1} \to \cdots \to S_{i_{p'-1}} \to S_{i(j)} \to S_{j_1} \to \cdots \to S_{j_{q'-1}} \to S_j \enspace , \]
where $j_k \in \mathcal{J}$ for all $k$. Denote by $q(j)$ the length $q'$ of the path from $S_{i(j)}$ to $S_j$, and let $q = \max_{j \in \mathcal{J}} q(j)$.

Clearly, if $y = G^q(x)$ and $j \in \mathcal{J}$, then $\pi_j(y)$ is a function of $\pi_{i(j)}(x)$, which in turn is a function of some $\pi_i(y)$ with $i \in \mathcal{I}$. But this means that the $\mathcal{J}$-components of $y$ are uniquely determined by its $\mathcal{I}$-components. This and the fact that $G$ is reversible on $S_\mathcal{I}^\Z$ imply that $X = G^q(S^\Z)$, $X$ is algebraically conjugate to $S_\mathcal{I}^\Z$, and $G^p|_X$ is a product of powers of shift maps. \qed
\end{proof}

\section{Future Work}

In this paper we have only considered limit sets in the case when the cellular automaton starts from the full shift. It would be interesting to study limit sets of $\F$-linear automata starting from more complicated shifts. We do not know if our approach generalizes to, say, mixing $\F$-subshifts of finite type, assuming the congruence-product property. Future work might also involve studying the connections between other properties of the variety $\F$ and the $\F$-linear CA.

\section*{Acknowledgements}

We would like to thank an anonymous referee of STACS 2012 for suggesting Proposition~\ref{prop:SoficLattices}.

\bibliographystyle{plain}
\bibliography{//utuhome.utu.fi/iatorm/bib/bib}{}

\end{document}